\title{Infinite Free Resolutions over Monomial Rings in Two Variables}
\author{
Gwyneth R. Whieldon
        }
\date{\today}
\documentclass[12pt]{article}

\usepackage{amsmath}
\usepackage{amsfonts}
\usepackage{amssymb}
\usepackage{amsthm}
\usepackage{amsopn}
\usepackage{tikz}
\usetikzlibrary{calc,patterns,matrix,shapes,snakes}
\usepackage{enumerate}

\theoremstyle{definition}
\newtheorem{theorem}{Theorem}[section]
\newtheorem{definition}[theorem]{Definition}
\newtheorem{lemma}[theorem]{Lemma}

\newtheorem{example}[theorem]{Example}
\newtheorem{proposition}[theorem]{Proposition}

\newtheorem{openquestion}[theorem]{Open Question}
\newtheorem*{notation*}{Notation}

\usepackage{thmtools}
\usepackage{thm-restate}

\newcommand{\NN}{{\mathbb N}}

\newcommand{\kk}{\Bbbk}

\newcommand{\Syz}{\text{Syz}}

\begin{document}


\maketitle



\begin{abstract}
Let \(M\subseteq\kk[x,y]\) be a monomial ideal \(M=(m_1,m_2,...,m_r)\), where the \(m_i\) are a minimal generating set of \(M\).  We construct an explicit free resolution of \(\kk\) over \(S=\kk[x,y]/M\) for all monomial ideals \(M\), and provide recursive formulas for the Betti numbers.  In particular, if \(M\) is any monomial ideal (excepting five degenerate cases), the total Betti numbers \(\beta_i^S(\kk)\) are given by \(\beta_0^S(\kk)=1\), \(\beta_1^S(\kk)=2\), and \(\beta_i^S(\kk)=\beta_{i-1}(\kk)+(r-1)\beta_{i-2}^S(\kk)\), where \(r\) is the number of minimal generators of \(M\).

This specializes to the classic example \(S=\kk[x,y]/(x^2,xy)\), which has \(\beta_i^S(\kk)=f_{i+1}\), where \(f_{i+1}\) is the (i+1)st Fibonacci number.
\end{abstract}



\section{Background and Motivation}

Let \(\kk\) be a field, \(R=\kk[x_1,x_2,...,x_n]\) a polynomial ring over \(\kk\), \(Q\) be an ideal of \(R\), and \(S=R/Q\) its quotient ring.  By the Auslander-Buchsbaum-Serre theorem, we have that the free resolution of a module \(A\) over \(S\) is finite if and only if \(S\) is a regular ring.  When \(S\) fails to be regular, however, the complexity of the infinite free resolution of \(A\) depends on the structure of \(S\).

For example, if \(S\) is a Golod ring, monomial ring or a generic toric ring, the Poincar\'e-Betti series is known to be rational.  See \cite{MR2641236}, \cite{MR2158760}, and \cite{MR686351} for results in this direction.  Examples of modules with Poincar\'e-Betti series which are \emph{not} rational are also known, even in the setting of toric rings (see \cite{MR644015}, \cite{MR733933}, and \cite{MR1646972}.)

In the setting of monomial rings \(S=R/M\), where \(M\) is a monomial ideal, the Poincar\'e-Betti series and (frequently non-minimal) resolutions of \(\kk\) over \(S\) are both known. Charalambous produced an explicit free resolution of a residue field \(\kk\) over monomial rings \(S=R/M\), where \(M\) is an arbitrary monomial ideal (Theorem 1, \cite{MR1382031}.)  These resolutions of $\kk$ over ring $S=R/M$ are not necessarily minimal (even for $R=\kk[x,y]$), but do contain the minimal resolutions produced in this paper as a direct summand.  From work of Backelin~\cite{MR686351} and Avramov~\cite{MR2641236}, in the case of monomial rings it is known that the multigraded Poincar\'e-Betti series is a rational function and depends exclusively on the field \(\kk\) and the lcm-lattice of the monomial ideal \(M\).  The denominator of this rational function has been interpreted combinatorially in \cite{MR2188858} and \cite{MR2290910}.

In this paper, we focus primarily on the case where \(R=\kk[x,y]\) and our ring \(S=R/M\) is the quotient of \(R\) by an monomial ideal \(M\).  We partition monomial ideals in two variables into six types (one main case and five degenerate cases,) and explicitly construct a minimal free resolution of \(\kk\) in each type. The denominator of the Poincar\'{e}-Betti series is given as an immediate corollary of the iterative construction of the syzygy modules.

In each of these cases, we use the recursively constructed free resolutions to produce the total Betti sequence \(\beta_i^S(k)\) for the residue fields of \(S=R/M\).  In particular, we completely classify which series can occur as the Poincar\'e-Betti sequence of \(\kk\) over a monomial ring in two variables.  Specifically, when \(M\) is a monomial ideal with \(r\geq 2\) generators (not both pure powers) the Betti sequence of the resolution of the residue class field is given by the sequence \(\beta_0=1\), \(\beta_1=2\), and \(\beta_n=\beta_{n-1}+(r-1)\beta_{n-2}\) for \(n\geq 2\).

%



Our main case includes all monomial ideals \(M\) of the form
\[M=(x^{a_1}y^{b_1},...,x^{a_r}y^{b_r}),\]
where the \(x^{a_i}y^{b_i}\) form a minimal generating set and \(r>2\), or \(r=2\) and not of the form \((x^a,y^b)\).  The five degenerate cases on two or fewer generators, each with distinct resolution types, are described in the Appendix in Section~\ref{sec:degenerate}.

\begin{restatable*}[Resolutions of \(\kk\) over Quotient Rings \(S\)]{theorem}{mainthm}\label{thm:main}  Let \(M\) be a monomial ideal with minimal generating set \(\{m_1,m_2,...,m_r\}\), where \(m_i=x^{a_i}y^{b_i}\) with \(a_1> a_2>\cdots>a_r\geq 0\) and \(0\leq b_1 <b_2<\cdots < b_r\) with \(r>2\), or \(r=2\) and \(M\) not generated by pure powers of \(x\) and \(y\).  A free resolution \({\mathcal F}\) of \(\kk\) over \(S\) is given by
\[{\mathcal F}:  S\xleftarrow{\;\;\partial_1\;\;} F_1\xleftarrow{\;\;\partial_2\;\;} F_2\xleftarrow{\;\;\partial_3\;\;} F_3\xleftarrow{\;\;\partial_4\;\;} F_4\longleftarrow \cdots\]
where \(F_1\cong S^2\), \(F_2\cong S^{r+1}\), and \(F_3\cong S^{3r-1}\).

For \(i\geq 3\), the \((i+1)^{\text{st}}\) stage of the resolution for \(i\geq 3\) is constructed recursively in the following way:

Let \(F_i=F_1^{u_i}\oplus F_2^{v_i}\oplus F_3^{w_i}\) and \(\partial_i=\partial_1^{u_i}\oplus\partial_2^{v_i}\oplus\partial_3^{w_i}\) (with fixed \(u_i,v_i,w_i\) constructed via this method at a previous stage). Then the \((i+1)^{\text{st}}\)-stage of the minimal resolution is given by
\[F_i\cong F_1^{u_i}\oplus F_2^{v_i}\oplus F_3^{w_i}\xleftarrow{\partial_{i+1}} F_1^{(r-1)w_i}\oplus F_2^{u_i+rw_i}\oplus F_3^{v_i}\cong F_{i+1}\]
with \(\partial_{i+1}=\partial_3^{v_i}\oplus \partial_2^{w_i+ru_i}\oplus \partial_1^{(r-1)u_i}\).  Explicit formulas for the maps \(\partial_1\), \(\partial_2\), and \(\partial_3\) are given in Propositions~\ref{prop:syz1and2},~\ref{prop:syz3}, and~\ref{prop:syz4}
\end{restatable*}

With an explicit minimal free resolution in hand, a description of the Poincar\'{e}-Betti series is immediate.

\begin{restatable*}[Poincar\'{e}-Betti Series of \(\kk\) over \(S\)]{theorem}{maincor}\label{cor:main}  Let \(M=(x^{a_1}y^{b_1},\cdots,x^{a_r}y^{b_r})\), with \(a_1> a_2>\cdots>a_r\geq 0\) and \(0\leq b_1 <b_2<\cdots < b_r\) with \(r>2\), or \(r=2\) and \(M\) not generated by pure powers of \(x\) and \(y\).  The total Betti numbers of the resolution of \(\kk\) over \(S=R/M\) are given by
\begin{equation*}
\beta_i^S(\kk)=\begin{cases}
1&\text{ if \(i=0\),}\\
2&\text{ if \(i=1\),}\\
\beta_{i-1}^S(\kk)+(r-1)\beta_{i-2}^S(\kk)&\text{ if \(i\geq 2\).}\\
\end{cases}
\end{equation*}
The Poincar\'e-Betti series then is
\[P_S(z)=\frac{1+z}{1-z+(1-r)z^2}.\]
\end{restatable*}

Note that this implies that the Poincar\'e-Betti sequences of these resolutions depends only on the number of generators of \(M\), rather than on the degrees of or relations between the generators.

\begin{example}[2-generated ideals]\label{example:main} Consider the monomial ideals \((xy^2,y^4)\) and \((x^2y,xy^2)\).  The graded Betti diagram of the minimal resolutions of \(\kk\) found over \(S_1=\kk[x,y]/(xy^2,y^4)\) and \(S_2=\kk[x,y]/(x^2y,xy^2)\) are:

\vspace{10pt}
\begin{center}
$\begin{matrix}
&0&1&2&3&4&5&6&\cdots\\
\text{total:}&1&2&3&5&8&13&21&\cdots\\
\text{0:}&1&2&1&\text{.}&\text{.}&\text{.}&\text{.}\\
\text{1:}&\text{.}&\text{.}&1&2&1&\text{.}&\text{.}\\
\text{2:}&\text{.}&\text{.}&1&3&4&3&1\\
\text{3:}&\text{.}&\text{.}&\text{.}&\text{.}&2&6&7\\
\text{4:}&\text{.}&\text{.}&\text{.}&\text{.}&1&4&9\\
\text{5:}&\text{.}&\text{.}&\text{.}&\text{.}&\text{.}&\text{.}&3\\
\text{6:}&\text{.}&\text{.}&\text{.}&\text{.}&\text{.}&\text{.}&1\\
      \end{matrix}\;\;\;\;\;\;\;\;\begin{matrix}
      &0&1&2&3&4&5&6&\cdots \\
      \text{total:}&1&2&3&5&8&13&21&\cdots \\
\text{0:}&1&2&1&\text{.}&\text{.}&\text{.}&\text{.}\\
\text{1:}&\text{.}&\text{.}&2&5&4&1&\text{.}\\
\text{2:}&\text{.}&\text{.}&\text{.}&\text{.}&4&12&13\\
\text{3:}&\text{.}&\text{.}&\text{.}&\text{.}&\text{.}&\text{.}&8\\
      \end{matrix}$
\end{center}
\vspace{10pt}

Note that while their graded Betti diagrams differ, the total Betti numbers are the same, given by \(\beta_i^{S_1}(\kk)=\beta_i^{S_2}(\kk)=f_{i+1}\), where \(f_{i+1}\) is the \((i+1)^{\text{st}}\) Fibonacci number.  The graded Betti numbers depend (in the main case) exclusively on the number and degrees of the generators \(m_i\).
\end{example}

\section{Staircase Diagram of Monomial Ideals in Two Variables}

Let \(M=(x^{a_1}y^{b_1},...,x^{a_r}y^{b_r})\) be a monomial ideal (given by its minimal generating set,) ordered such that \(a_1> a_2>\cdots>a_r\geq 0\) and \(0\leq b_1 <b_2<\cdots < b_r\).  A monomial ideal \(M\) in two variables can always be put into this form, which may be represented by a \emph{staircase diagram} (see \cite{MR2110098}). 
\begin{figure}[h!]
\begin{center}
\begin{tikzpicture}[scale=0.5,auto=left,vertices/.style={circle, fill=black, inner sep=1.5pt}]

\foreach \xcoord/\ycoord in {1/2,0/4} {
\fill[-,fill=black!20] (\xcoord,\ycoord)--(\xcoord,6)--(6,6)--(6,\ycoord)--(\xcoord,\ycoord);
\fill[-,fill=black!10] (\xcoord,6)--(6,6)--(6,\ycoord)--(6.5,\ycoord)--(6.5,6.5)--(\xcoord,6.5)--(\xcoord,6);
};

\foreach \yval in {1,2,3,4,5,6} {
	\draw[-,black!40] (0,\yval)--(6.5,\yval);
};
\foreach \xval in {1,2,3,4,5,6} {
	\draw[-,black!40] (\xval,0)--(\xval,6.5);
};

\draw[->] (0,0)--(7,0);
\draw[->] (0,0)--(0,7);
\node[label=right:{x}] at (7,0) {};
\node[label=above:{y}] at (0,7) {};

\foreach \xcoord/\ycoord in {1/2,0/4} {
	\node[vertices] at (\xcoord,\ycoord){};
};

\node at (3.5,-1.5) {$M_1=(xy^2,y^4)$};
\end{tikzpicture}\;\;\;\;\;\;\begin{tikzpicture}[scale=0.5,auto=left,vertices/.style={circle, fill=black, inner sep=1.5pt}]

\foreach \xcoord/\ycoord in {2/1,1/2} {
\fill[-,fill=black!20] (\xcoord,\ycoord)--(\xcoord,6)--(6,6)--(6,\ycoord)--(\xcoord,\ycoord);
\fill[-,fill=black!10] (\xcoord,6)--(6,6)--(6,\ycoord)--(6.5,\ycoord)--(6.5,6.5)--(\xcoord,6.5)--(\xcoord,6);
};

\foreach \yval in {1,2,3,4,5,6} {
	\draw[-,black!40] (0,\yval)--(6.5,\yval);
};
\foreach \xval in {1,2,3,4,5,6} {
	\draw[-,black!40] (\xval,0)--(\xval,6.5);
};

\draw[->] (0,0)--(7,0);
\draw[->] (0,0)--(0,7);
\node[label=right:{x}] at (7,0) {};
\node[label=above:{y}] at (0,7) {};

\foreach \xcoord/\ycoord in {2/1,1/2} {
	\node[vertices] at (\xcoord,\ycoord){};
};

\node at (3.5,-1.5) {\(M_2=(x^2y,xy^2)\)};
\end{tikzpicture}
\end{center}

\caption{Staircase diagrams of \(M_1\) and \(M_2\) shown.\label{fig:staircase}}

\end{figure}

See Figure~\ref{fig:staircase} for the staircase diagrams of the monomial ideals in Example~\ref{example:main}. All monomials \(m\) corresponding to lattice points inside or on the boundary of the grey area are in the ideal \(M\), and all of the lattice points outside of the shaded area or its boundary are monomials in the ring \(S=\kk[x,y]/M\). 

\begin{definition}[Colon Ideals $M_x$ and $M_y$ in $S$]\label{defn:colon}Let \(M_x\) and \(M_y\) be the ideals \(0:(x)\) and \(0:(y)\) of $S$,
\begin{align*}
M_x &= \{ m : mx\in M, m\not\in M\}=0:(x),\text{ and}\\
M_y &= \{ m : my\in M, m\not\in M\}=0:(y).
\end{align*}
\end{definition}
The elements of \(M_x\) and \(M_y\) for the monomial ideals in Example~\ref{example:main} are shown in white in Figure~\ref{fig:halo}.

\begin{figure}[h!]
\begin{center}
\begin{tikzpicture}[scale=0.4,auto=left,vertices/.style={circle, draw, fill=black, inner sep=1.5pt}, dotvertices/.style={circle, draw, fill=white, inner sep=1.5pt}]

\foreach \xcoord/\ycoord in {1/2,0/4} {
\fill[-,fill=black!20] (\xcoord,\ycoord)--(\xcoord,6)--(6,6)--(6,\ycoord)--(\xcoord,\ycoord);
\fill[-,fill=black!10] (\xcoord,6)--(6,6)--(6,\ycoord)--(6.5,\ycoord)--(6.5,6.5)--(\xcoord,6.5)--(\xcoord,6);
};

\foreach \yval in {1,2,3,4,5,6} {
	\draw[-,black!40] (0,\yval)--(6.5,\yval);
};
\foreach \xval in {1,2,3,4,5,6} {
	\draw[-,black!40] (\xval,0)--(\xval,6.5);
};

\draw[->] (0,0)--(7,0);
\draw[->] (0,0)--(0,7);

\foreach \xcoord/\ycoord in {1/2,0/4} {
\node[vertices] at (\xcoord,\ycoord){};
};

\foreach \xcoord/\ycoord in {0/2,0/3} {
\node[dotvertices] at (\xcoord,\ycoord){};
\draw[->] (\xcoord+0.1,\ycoord)--(\xcoord+0.8,\ycoord);
};

\node[label=right:{x}] at (7,0) {};
\node[label=above:{y}] at (0,7) {};

\node at (3.5,-1.5) {\(M_1:(x)=(y^2)\)};
\end{tikzpicture}\;\;\;\;\;\;\begin{tikzpicture}[scale=0.4,auto=left,vertices/.style={circle, draw, fill=black, inner sep=1.5pt},dotvertices/.style={circle, draw, fill=white, inner sep=1.5pt}]

\foreach \xcoord/\ycoord in {2/1,1/2} {
\fill[-,fill=black!20] (\xcoord,\ycoord)--(\xcoord,6)--(6,6)--(6,\ycoord)--(\xcoord,\ycoord);
\fill[-,fill=black!10] (\xcoord,6)--(6,6)--(6,\ycoord)--(6.5,\ycoord)--(6.5,6.5)--(\xcoord,6.5)--(\xcoord,6);
};

\foreach \yval in {1,2,3,4,5,6} {
	\draw[-,black!40] (0,\yval)--(6.5,\yval);
};
\foreach \xval in {1,2,3,4,5,6} {
	\draw[-,black!40] (\xval,0)--(\xval,6.5);
};

\draw[->] (0,0)--(7,0);
\draw[->] (0,0)--(0,7);

\foreach \xcoord/\ycoord in {2/1,1/2} {
\node[vertices] at (\xcoord,\ycoord){};
};

\foreach \xcoord/\ycoord in {1/1,0/2,0/3,0/4,0/5,0/6} {
\node[dotvertices] at (\xcoord,\ycoord){};
\draw[->] (\xcoord+0.1,\ycoord)--(\xcoord+0.8,\ycoord);
};

\node[label=right:{x}] at (7,0) {};
\node[label=above:{y}] at (0,7) {};

\node at (3.5,-1.5) {\(M_2:(x)=(xy,y^2)\)};
\end{tikzpicture}

\begin{tikzpicture}[scale=0.4,auto=left,vertices/.style={circle, draw, fill=black, inner sep=1.5pt}, dotvertices/.style={circle, draw, fill=white, inner sep=1.5pt}]

\foreach \xcoord/\ycoord in {1/2,0/4} {
\fill[-,fill=black!20] (\xcoord,\ycoord)--(\xcoord,6)--(6,6)--(6,\ycoord)--(\xcoord,\ycoord);
\fill[-,fill=black!10] (\xcoord,6)--(6,6)--(6,\ycoord)--(6.5,\ycoord)--(6.5,6.5)--(\xcoord,6.5)--(\xcoord,6);
};

\foreach \yval in {1,2,3,4,5,6} {
	\draw[-,black!40] (0,\yval)--(6.5,\yval);
};
\foreach \xval in {1,2,3,4,5,6} {
	\draw[-,black!40] (\xval,0)--(\xval,6.5);
};

\draw[->] (0,0)--(7,0);
\draw[->] (0,0)--(0,7);

\foreach \xcoord/\ycoord in {1/2,0/4} {
\node[vertices] at (\xcoord,\ycoord){};
};

\foreach \xcoord/\ycoord in {0/3,1/1,2/1,3/1,4/1,5/1,6/1} {
\node[dotvertices] at (\xcoord,\ycoord){};
\draw[->] (\xcoord,\ycoord+0.1)--(\xcoord,\ycoord+0.8);
};

\node[label=right:{x}] at (7,0) {};
\node[label=above:{y}] at (0,7) {};

\node at (3.5,-1.5) {\(M_1:(y)=(xy,y^3)\)};
\end{tikzpicture}\;\;\;\;\;\;\begin{tikzpicture}[scale=0.4,auto=left,vertices/.style={circle, draw, fill=black, inner sep=1.5pt},dotvertices/.style={circle, draw, fill=white, inner sep=1.5pt}]

\foreach \xcoord/\ycoord in {2/1,1/2} {
\fill[-,fill=black!20] (\xcoord,\ycoord)--(\xcoord,6)--(6,6)--(6,\ycoord)--(\xcoord,\ycoord);
\fill[-,fill=black!10] (\xcoord,6)--(6,6)--(6,\ycoord)--(6.5,\ycoord)--(6.5,6.5)--(\xcoord,6.5)--(\xcoord,6);
};

\foreach \yval in {1,2,3,4,5,6} {
	\draw[-,black!40] (0,\yval)--(6.5,\yval);
};
\foreach \xval in {1,2,3,4,5,6} {
	\draw[-,black!40] (\xval,0)--(\xval,6.5);
};

\draw[->] (0,0)--(7,0);
\draw[->] (0,0)--(0,7);

\foreach \xcoord/\ycoord in {2/1,1/2} {
\node[vertices] at (\xcoord,\ycoord){};
};

\foreach \xcoord/\ycoord in {1/1,2/0,3/0,4/0,5/0,6/0} {
\node[dotvertices] at (\xcoord,\ycoord){};
\draw[->] (\xcoord,\ycoord+0.1)--(\xcoord,\ycoord+0.8);
};

\node[label=right:{x}] at (7,0) {};
\node[label=above:{y}] at (0,7) {};

\node at (3.5,-1.5) {\(M_2:(y)=(x^2,xy)\)};
\end{tikzpicture}
\end{center}
\caption{Elements of \(M_x\) and \(M_y\). \label{fig:halo}}

\end{figure}

We will use the ideals \(M_x\) and \(M_y\) frequently when constructing our syzygies of \(\kk\) over \(S=\kk[x,y]/M\).

\begin{proposition}[First and Second Syzygy Modules of \(\kk\) over \(S\)]\label{prop:syz1and2}
Let \(M=(x^{a_1}y^{b_1},...,x^{a_r}y^{b_r})\) be a monomial ideal given by its minimal set of generators, ordered so that \(a_1> a_2>\cdots>a_r\geq 0\) and \(0\leq b_1 <b_2<\cdots < b_r\).  Let either \(r>2\) or \(r=2\) and \(M\neq(x^n,y^m)\).  The first three stages of a (graded) minimal resolution of \(\kk\) over \(S=\kk[x,y]/M\) are given by

\[{\mathcal F}: \kk \longleftarrow F_0 \xleftarrow{\;\;\;\partial_1\;\;\;} F_1\xleftarrow{\;\;\;\partial_2\;\;\;} F_2.\]
These syzygy modules and their respective bases are:
\begin{center}
\begin{tabular}{l c l}
Syzygy Module &\;\;\; & Basis\\ \hline
\(F_0\cong S\)&\;\;\;& \(\{e_1\}\)\\
\(F_1=S(-1)^2\cong S^2\)&\;\;\;&\(\{e_x,e_y\}\)\\
\(F_2=\left(\bigoplus_{1\leq i\leq r} S(-a_i-b_i)\right)\oplus S(-2)\cong S^{r+1}\)&\;\;\;&\(\{e_{f_1},e_{f_2},...,e_{f_r},e_{f_{r+1}}\}\)
\end{tabular}
\end{center}

\noindent The map \(F_1\xrightarrow{\partial_1} F_2\) is given by 
\begin{align*}
\partial_1(e_x)&=x\cdot e_1\\
\partial_1(e_y)&=y\cdot e_1.
\end{align*}
The map \(F_2\xrightarrow{\partial_2}F_1\) is given by:

\begin{enumerate}[{\bf Case }1:]
\item If all generators of \(M\) are divisible by \(x\) (i.e. \(a_r\geq 1\)), then the second syzygy map \(\partial_2\) is given by
\begin{equation*}
\partial_2(e_{f_i})=\begin{cases}
x^{a_{i}-1}y^{b_{i}}\cdot e_x & \text{if \(1\leq i\leq r\),}\\
-y\cdot e_x+x\cdot e_y&\text{if \(i=r+1\).}
\end{cases}
\end{equation*}
\item If the final generator of \(M\) is \(m_r=y^{b_r}\), then the second syzygy map \(\partial_2\) is given by
\begin{equation*}
\partial_2(e_{f_i})=\begin{cases}
x^{a_{i}-1}y^{b_{i}}\cdot e_x & \text{if \(1\leq i\leq r-1\),}\\
y^{b_r-1}\cdot e_y & \text{if \(i=r\),}\\
-y\cdot e_x+x\cdot e_y&\text{if \(i=r+1\).}
\end{cases}
\end{equation*}
\end{enumerate}

\end{proposition}

\begin{proof}
Let \(M\) be a monomial ideal satisfying the conditions of the theorem.  Note that \(r>2\) or \(r=2\) and \(M\neq(x^n,y^m)\) for \(n,m\geq 0\) implies that \(x,y\) are nonzero elements in \(S\).  Both \(x,y\in S\) will map to zero in the residue field \(\kk\), so \(x\) and \(y\) must span our set of syzygies in \(F_0=S\). We choose \(e_1\) as a generator of \(F_0\), and \(e_x,e_y\) as generators for our first syzygy module \(F_1\cong S(-1)^2\), and we define \(\partial_1(e_x)=x\cdot e_1\) and \(\partial_1(e_y)=y\cdot e_1\).

To construct the second syzygy module, we note that we are looking for a generating set for all elements \(w\cdot e_x+z\cdot e_y\in F_1\) such that \(wx+zy=0\in S\).

By construction of \(M_x=\{ m : mx\in M, m\not\in M\}=0:(x)\) and \(M_y=\{ m : my\in M, m\not\in M\}=0:(y)\) in the previous section, we have that \(\{m\cdot e_x : m\in M_x\}\) and \(\{m'\cdot e_y : m'\in M_y\}\) are nonzero elements of \(F_1\) in the kernel of \(\partial_1\). We also have that \(-yx+xy=0\), so \(-y\cdot e_x+x\cdot e_y\) is another element of \(\ker(\partial_1)\). We now show that all elements in
\begin{align*}
\ker(\partial_1)&=\{w\cdot e_x+z\cdot e_y : wx+zy=0\}\\
&=\{w\cdot e_x+z\cdot e_y : wx=-zy\in S\}
\end{align*}
 may be written as a linear combination of syzygies of these three types:
 \[\underbrace{\left\{x^{a_i-1}y^{b_i}\cdot e_x\right\}_{i=1}^r}_{\text{Type I}}\bigcup\underbrace{\left\{x^{a_i}y^{b_i-1}\cdot e_y\right\}_{i=1}^r}_{\text{Type II}}\bigcup \underbrace{\left\{-y\cdot e_x+x\cdot e_y\right\}}_{\text{Type III}}\]
 

Given a syzygy \(w\cdot e_x+z\cdot e_y\), it follows that \(wx=-zy\in S\). This implies that either
\begin{enumerate}[(i)]
\item \(wx=zy=0\in S\) or 
\item \(wx=-zy=s\in S\), where \(s\neq 0\in S\).
\end{enumerate}

If condition (i) holds, we have \(wx,zy\in M\), which implies that \(w\in M_x\) and \(z\in M_y\).  So \(w\cdot e_x+z\cdot e_y\) is a linear combinations of syzygies of the first two types.

If condition (ii) holds, we can factor out a \(y\) from \(w\) and an \(x\) from \(z\), giving
\begin{align*}
s & = wx=-zy\\
& = w'xy=-z'xy.
\end{align*}
So \(s\) is a monomial divisible by \(xy\), and our original syzygy can be obtained as a multiple of the syzygy \(-y \cdot e_x+x\cdot e_y\),
\begin{align*}
-\frac{s}{xy}(-y\cdot e_x+x\cdot e_y)&=\frac{s}{x}\cdot e_x-\frac{s}{y}\cdot e_y=w\cdot e_x+z\cdot e_y.
\end{align*}
So every syzygy of \(F_1\) can be written as a linear combination of syzygies of the three types above. We now reduce this set of syzygies to a minimal generating set.

Assume \(b_i\geq 1\), so the syzygy \(x^{a_i}y^{b_i-1}\cdot e_y\) exists. If \(b_i=0\), there are no syzygies to consider. If \(a_i\geq 1\), we may write
\begin{align*}
x^{a_i}y^{b_i-1}\cdot e_y & = \left(x^{a_i-1}y^{b_i}-x^{a_i-1}y^{b_i}\right)\cdot e_x+x^{a_i}y^{b_i-1}\cdot e_y\\
& = x^{a_i-1}y^{b_i}\cdot e_x+\left(-x^{a_i-1}y^{b_i}\cdot e_x+x^{a_i}y^{b_i-1}\cdot e_y\right)\\
& = x^{a_i-1}y^{b_i}\cdot e_x+x^{a_i-1}y^{b_i-1}\left(-y\cdot e_x+x\cdot e_y\right)
\end{align*}
Note that this syzygy is a linear combination of syzygies of the first and third type.
In Case 1, \(a_i\geq 1\) for all generators \(x^{a_i}y^{b_i}\) of \(M\). We may then remove all syzygies of Type II, \(m'\cdot e_y\), from our basis, rewriting them as a linear combination of syzygies of the first and third type.  This gives us a kernel minimally spanned by 
\[\left\{x^{a_i-1}y^{b_i}\cdot e_x\right\}_{i=1}^r\bigcup \left\{-y\cdot e_x+x\cdot e_y\right\}.\]

In Case 2, we have \(a_r=0\), and our kernel is spanned by 
\[\left\{x^{a_i-1}y^{b_i}\cdot e_x\right\}_{i=1}^{r-1}\bigcup\left\{y^{b_r-1}\cdot e_y\right\}\bigcup\left\{-y\cdot e_x+x\cdot e_y\right\}.\]
This is also minimal, as the \(x^{a_i-1}y^{b_i}\) are the minimal generating set of \(M_x\), and by construction, \(y^{b_r-1}\cdot e_y\) cannot be written as a linear combination of any of the other syzygies.

Our basis for \(F_2\) is then given by \(\{e_{f_1},e_{f_2},...,e_{f_r},e_{f_{r+1}}\}\), with \(\partial_2(e_{f_i})\) given as in the statement of the theorem, with the degree-shifts of \(F_2\) calculated appropriately.

\end{proof}

We verify that by construction, \(\partial_1\circ\partial_2(e_{f_i})=0\), so
\[F_0\xleftarrow{\;\;\;\partial_1\;\;\;} F_1\xleftarrow{\;\;\;\partial_2\;\;\;}F_2\] as defined form a complex:
\begin{equation*}
\partial_1\circ \partial_2(e_{f_i})=\begin{cases}
\partial_1(x^{a_i-1}y^{b_i}\cdot e_x)&= x^{a_i-1}y^{b_i}\cdot \partial_1(e_x)\\
&=x^{a_i}y^{b_i}\cdot e_1=0 \hspace{51pt} \text{ if \(a_i\geq 1\) and \(1\leq i\leq r\),}\\
\partial_1(y^{b_r-1}\cdot e_y)&= y^{b_r-1}\cdot \partial_1(e_y)\\
&=y^{b_r}\cdot e_1=0 \hspace{65pt}\text{ if \(a_r\geq 0\) and \(i=r\),}\\
\partial_1(-y\cdot e_x+x\cdot e_y)&=-y\cdot \partial_1(e_x)+x\cdot \partial_1(e_y)\\
&=(-yx+xy)\cdot e_1=0 \hspace{25pt} \text{if \(i=r+1\).}\\
\end{cases}
\end{equation*}

\section{Third Syzygy Modules in Main Case}

For an \(r\)-generated monomial ideal \(M\) in \(x,y\), we have so far that the resolution of \(\kk\) over \(S=\kk[x,y]/M\) has second syzygy module \[F_2=\left(\bigoplus_{1\leq i\leq r} S(-a_i-b_i)\right)\oplus S(-2)\cong S^{r+1}\] 
with basis \(\{e_{f_1},e_{f_2},...,e_{f_r},e_{f_{r+1}}\}\), and a second syzygy map:

\begin{equation*}
\partial_2(e_{f_i})=\begin{cases}
x^{a_{i}-1}y^{b_{i}}\cdot e_x & \text{if \(a_i\geq 1\) and \(1\leq i\leq r\),}\\
y^{b_r-1}\cdot e_y & \text{if \(a_r=0\) and \(i=r\),}\\
-y\cdot e_x+x\cdot e_y&\text{if \(i=r+1\)}.
\end{cases}
\end{equation*}

\noindent Before constructing the third syzygy module, we begin with a useful lemma describing resolutions of ideal $M_x$ over $S$.

\begin{lemma}[Syzygy Module of $M_x$ over $S$]\label{lem:syzMx} Let $M=(x^{a_1}y^{b_1},...,x^{a_r}y^{b_r})\subset \kk[x,y]=R$ where $0\leq b_1<\cdots<b_r$ and $a_1>\cdots > a_r\geq 0$, and let $S=R/M$. Let $M_x=0:(x)$ be the module over $S$ from Definition~\ref{defn:colon}.
Let $\{e_{f_i}\}$ be the minimal generating set for $M_x$ as a module over $S$, and set $M_x\xrightarrow{\partial_0} S$ where $\partial_0(e_{f_i})=x^{a_i-1}y^{b_i}\in S$.
Then the first syzygy module of $M_x$, $\Syz_1(M_x)\cong\ker(\partial_0)$, is minimally generated by
\[\left\{e_{g_i}\right\}_{i=1}^{r}\cup\left\{e_{g_i'}\right\}_{i=1}^{r-1}\text{ if $a_r>0$,}\]
\[\left\{e_{g_i}\right\}_{i=1}^{r-1}\cup\left\{e_{g_i'}\right\}_{i=1}^{r-1}\text{ if $a_r=0$,}\]
where
\begin{align*}
\left\{\partial_1(e_{g_i})=x\cdot e_{f_i}\right\}_{i=1}^{r}&\cup\left\{\partial_1(e_{g_i'})=y^{b_{i+1}-b_{i}}\cdot e_{f_i}\right\}_{i=1}^{r-1}& \text{ for $a_r>0$}\;\\
\left\{\partial_1(e_{g_i})=x\cdot e_{f_i}\right\}_{i=1}^{r-1}&\cup\left\{\partial_1(e_{g_i'})=y^{b_{i+1}-b_{i}}\cdot e_{f_i}\right\}_{i=1}^{r-1}& \text{ for $a_r=0.$}
\end{align*}
\end{lemma}

\begin{proof}
Note that $M_x$ is minimally generated by the following monomials:
\begin{equation*}
M_x=\begin{cases}
(x^{a_1-1}y^{b_1},\ldots, x^{a_r-1}y^{b_r}) & \text{ if $a_r>0$,}\\
(x^{a_1-1}y^{b_1},\ldots, x^{a_{r-1}-1}y^{b_{r-1}}) & \text{ if $a_r=0$.}
\end{cases}
\end{equation*}
As $M_x$ is a monomial ideal over a monomial quotient ring, all sygyzies must take one of two forms:
\begin{align*}
m_i\cdot x^{a_i-1}y^{b_i}-m_j\cdot x^{a_j-1}y^{b_j}&=0\text{ for $m_i,m_j\neq 0\in S$}\;\;\;&\text{(Type I)}\;\;\\
m_i\cdot x^{a_i-1}y^{b_i}&=0\text{ for $m_i\neq 0\in S$}\;\;\;\;\;&\text{(Type II)}.
\end{align*}
As $x\cdot x^{a_i-1}y^{b_i}=0$ in $S$ for all generators $x^{a_i-1}y^{b_i}\in M_x$, we add the elements $\left\{e_{g_i}:\genfrac{}{}{0pt}{}{1\leq i\leq r \text{ if $a_r>0$}}{1\leq i\leq r-1 \text{ if $a_r=0$}}\right\}$ to a basis for $\Syz_1(M_x)$ with syzygy map $\partial_1(e_{g_i})=x\cdot e_{f_i}$. All syzygies of Type II with $x|m_i$ will be in the $S$-span of the image of $\partial_1$.

Consider any remaining syzygies of Type II not in this span, i.e. syzygies with $y^{d_i}\cdot x^{a_i-1}y^{b_i}=0$ where $x^{a_i-1}y^{b_i+(d_i-1)}\neq 0\in S$. Note that then $x^{a_i-1}y^{b_i+d_i}\in M$, and must in particular be divisible by $x^{a_i-1}y^{b_{i+1}}$. So we must have $y^{d_i}=y^{b_{i+1}-b_i}$, giving us new minimal syzygies of the forms $y^{b_{i+1}-b_i}\cdot e_{f_i}.$

\begin{figure}[h!]
\begin{center}
\begin{tikzpicture}[scale=0.4,auto=left,vertices/.style={circle, draw, fill=black, inner sep=1pt}, dotvertices/.style={circle, draw, fill=white, inner sep=1.5pt}]

\foreach \xcoord/\ycoord in {5/1,1/4} {
\fill[-,fill=black!15] (\xcoord,\ycoord)--(\xcoord,6)--(8,6)--(8,\ycoord)--(\xcoord,\ycoord);
\fill[-,fill=black!5] (\xcoord,6)--(8,6)--(8,\ycoord)--(8.5,\ycoord)--(8.5,6.5)--(\xcoord,6.5)--(\xcoord,6);
};

\foreach \yval in {0,1,2,3,4,5,6} {
	\draw[-,black!40] (0,\yval)--(8.5,\yval);
};
\foreach \xval in {0,1,2,3,4,5,6,7,8} {
	\draw[-,black!40] (\xval,0)--(\xval,6.5);
};

\foreach \xcoord/\ycoord in {6.5/3}{
\draw[black!15, fill=black!15, thick,rounded corners] (\xcoord-1.3,\ycoord-1.5) rectangle (\xcoord+1.3,\ycoord-0.5);
}

\foreach \xcoord/\ycoord in {4/5.9}{
\draw[black!15, fill=black!15, thick,rounded corners] (\xcoord-2.5,\ycoord-1.5) rectangle (\xcoord+1.3,\ycoord-0.5);
}

\foreach \xcoord/\ycoord in {2.5/3.4}{
\draw[white, fill=white, thick,rounded corners] (\xcoord-1.4,\ycoord-1.3) rectangle (\xcoord+1.5,\ycoord-0.7);
}

\foreach \xcoord/\ycoord in {2/1.6}{
\draw[white, fill=white, thick,rounded corners] (\xcoord-1.6,\ycoord-1.5) rectangle (\xcoord+1.6,\ycoord-0.5);
}

\draw[->] (4,1)--(4.8,1);
\node at (4.5,0.5){\tiny $x$};
\draw[->] (4,1)--(4,3.9);
\node at (2.5,2.4){\tiny $y^{b_{i+1}-b_i}$};

\node[vertices,label=above right:{\footnotesize $x^{a_i}y^{b_i}$}] at (5,1){};
\node[vertices,label=above right:{\footnotesize $x^{a_{i+1}}y^{b_{i+1}}$}] at (1,4){};
\node[dotvertices] at (4,1){};
\node at (2,0.6){\footnotesize $x^{a_i-1}y^{b_i}$};
\end{tikzpicture}
\end{center}
\caption{Syzygies of Type I on $M_x$ over $S$}\label{fig:syzygiesMx}
\end{figure}
Considering syzygies of Type I, we note that $\deg(m_i),\deg(m_j)\geq 1$. If $x|m_i$ or $x|m_j$, then we may rewrite one or both terms as sums of syzygies of Type II. The remaining term (if either $m_i$ or $m_j$ was a pure power of $y$) must be a syzygy in the $S$-span of $y^{b_{i+1}-b_i}\cdot e_{f_i}$. Hence all syzygies of $M_x$ must be in the $S$-span of the forms described in the statement of the Lemma.
\end{proof}

We now construct our third syzygy module $F_3$ over \(S\).

\begin{proposition}[Third Syzygy Module of \(\kk\) over \(S\)]\label{prop:syz3}  Let \(M=(x^{a_1}y^{b_1},...,x^{a_r}y^{b_r})\) be a monomial ideal given by its minimal set of generators, ordered so that \(a_1> a_2>\cdots>a_r\geq 0\) and \(0\leq b_1 <b_2<\cdots < b_r\).  Let either \(r>2\) or \(r=2\) and \(M\neq(x^n,y^m)\).  The third syzygy module of \(\kk\) over \(S=\kk[x,y]/M\) is of dimension \(3r-1\) with basis 
\[\bigl\{e_{c_i^x}\bigr\}_{i=1}^r\cup\bigl\{e_{c_i^y}\bigr\}_{i=1}^r\cup\bigl\{e_{d_i}\bigr\}_{i=1}^{r-1}.\]
\begin{enumerate}[{\bf Case }1:]
\item If all generators of \(M\) are divisible by \(x\) (i.e. \(a_r\geq 1\)), then the third syzygy map \(\partial_3\) is given by
\begin{align*}
\partial_3(e_{c_i^x})&=x\cdot e_{f_i} \hspace{95pt} \text{ for \(1\leq i\leq r\)}\\
\partial_3(e_{c_i^y})&=y\cdot e_{f_i}+x^{a_i-1}y^{b_i}\cdot e_{f_{r+1}} \hspace{10pt}\text{ for \(1\leq i\leq r\)}\\
\partial_3(e_{d_i})&=x^{a_i-1}y^{b_{i+1}-1}\cdot e_{f_{r+1}}\hspace{32pt}\text{ for \(1\leq i\leq r-1\)}.
\end{align*}
\item If the final generator of \(M\) is \(m_r=y^{b_r}\), then the third syzygy map \(\partial_3\) is given by
\begin{align*}
\partial_3(e_{c_i^x})&=\begin{cases} x\cdot e_{f_i} &\hspace{15pt}\text{for \(1\leq i\leq r-1\)}\\
x\cdot e_{f_r}-y^{b_r-1}\cdot e_{f_{r+1}} &\hspace{15pt}\text{for \(i=r\),}
\end{cases}\\
\partial_3(e_{c_i^y})&=\begin{cases}y\cdot e_{f_i}+x^{a_i-1}y^{b_i}\cdot e_{f_{r+1}} &\hspace{3pt} \text{for \(1\leq i\leq r-1\)}\\
y\cdot e_{f_{r}} &\hspace{3pt} \text{for \(i=r\),}
\end{cases}\\
\partial_3(e_{d_i})&=x^{a_i-1}y^{b_{i+1}-1}\cdot e_{f_{r+1}} \hspace{43pt}\text{ for \(1\leq i\leq r-1\)}.
\end{align*}
\end{enumerate}
\end{proposition}

We wish to construct a generating set for all syzygies
$$g_1\cdot e_{f_1}+ g_2\cdot e_{f_2}+\cdots+g_re_{f_r}+g_{r+1}e_{f_{r+1}},$$
for $g_1,...,g_r,g_{r+1}\in S$, where any given syzygy must satisfy the equations
\begin{align}
\label{eqn:1syz3}\text{Case }1:\;\;\;\;& \left({\displaystyle \sum_{i=1}^{r}g_ix^{a_{i}-1}y^{b_i}}\right)-g_{r+1}y=0 \\
\label{eqn:2syz3}& \;\;\;\;\;\;\;\;\;\;\;\;\;\;\;\;\;\;\;\;\;\;\;\;\;\;\;\;\;\;\,g_{r+1}x=0\\
\label{eqn:3syz3}\text{Case }2:\;\;\;\;& \left({\displaystyle \sum_{i=1}^{r-1}g_ix^{a_{i}-1}y^{b_i}}\right)-g_{r+1}y=0 \\
\label{eqn:4syz3}& \;\;\;\;\;\;\;\;\;\;\;\;\;\;\;\,g_{r}y^{b_r-1}+g_{r+1}x=0.
\end{align}

\begin{proof}[Proof of Case 1]
In Case 1, Equation~\ref{eqn:2syz3} forces $g_{r+1}=0$ or $g_{r+1}\in M_x.$

\noindent If $g_{r+1}=0$, it suffices to find a basis for syzygies 
$$g_1\cdot e_{f_1}+ g_2\cdot e_{f_2}+\cdots+g_re_{f_r},$$
where $${\displaystyle \sum_{i=1}^{r}g_ix^{a_{i}-1}y^{b_i}}=0.$$
By Lemma~\ref{lem:syzMx}, we know that a basis for syzygies of this form is given by
$$\left\{x\cdot e_{f_i}\right\}_{i=1}^r\cup\left\{y^{b_{i+1}-b_i}\cdot e_{f_i}\right\}_{i=1}^{r-1}.$$
We will rewrite the syzygies $y^{b_{i+1}-b_i}\cdot e_{f_i}$ as sums of other basis elements later in the proof.

If $g_{r+1}\neq 0$, then $g_{r+1}\in M_x$. As these syzygies are the kernel of a map with monomial entries over a monomial quotient $S=R/M$, we must have minimal syzygies of the form
$$g_ix^{a_i-1}y^{b_i}-g_{r+1}y=0$$
for some $g_i\in S$ and $g_{r+1}\in M_x$.

If $g_i=0$, this forces $g_{r+1}\in M_x\cap M_y$, as then $x\cdot g_{r+1}=y\cdot g_{r+1}=0$. So $g_{r+1}=x^{a^i-1}y^{b_{i+1}-1}$. We have one such syzygy for each $g_{r+1}\in M_x\cap M_y$.

\begin{figure}[h!]
\begin{center}
\begin{tikzpicture}[scale=0.4,auto=left,vertices/.style={circle, draw, fill=black, inner sep=0.75pt}, dotvertices/.style={circle, draw, fill=white, inner sep=1pt},bigdotverticesy/.style={star, star points=4, draw, fill=black!30, inner sep=1.5pt},dotverticesy/.style={star, star points=4, draw, fill=black!30, inner sep=0.75pt}]

\foreach \xcoord/\ycoord in {1/12,2/11,5/10,6/8,10/6,13/1} {
\fill[-,fill=black!15] (\xcoord,\ycoord)--(\xcoord,14)--(16,14)--(16,\ycoord)--(\xcoord,\ycoord);
\fill[-,fill=black!5] (\xcoord,14)--(16,14)--(16,\ycoord)--(16.5,\ycoord)--(16.5,14.5)--(\xcoord,14.5)--(\xcoord,14);
};

\foreach \yval in {0,1,2,3,4,5,6,7,8,9,10,11,12,13,14} {
	\draw[-,black!40] (0,\yval)--(16.5,\yval);
};
\foreach \xval in {0,1,2,3,4,5,6,7,8,9,10,11,12,13,14,15,16} {
	\draw[-,black!40] (\xval,0)--(\xval,14.5);
};

\foreach \xcoord/\ycoord in {1/12,2/11,5/10,6/8,10/6,13/1} {
\node[vertices] at (\xcoord,\ycoord){};
}

\foreach \xcoord/\ycoord in {1/11,4/10,5/9,9/7,12/5}{
\node[dotvertices] at (\xcoord,\ycoord){};
\node[bigdotverticesy] at (\xcoord,\ycoord){};
}

\foreach \xcoord/\ycoord in {0/12,0/13,0/14,1/11,4/10,5/9,5/8,9/7,9/6,12/5,12/4,12/3,12/2,12/1}{
\node[dotvertices] at (\xcoord,\ycoord){};
}

\foreach \xcoord/\ycoord in {2/10,3/10,6/7,7/7,8/7,10/5,11/5,13/0,14/0,15/0,16/0}{
\node[dotverticesy] at (\xcoord,\ycoord){};
}

\node[dotvertices, label=right:{\footnotesize Elements of $M_x$}] at (18,8){};
\node[dotverticesy, label=right:{\footnotesize Elements of $M_y$}] at (18,6.5){};
\node[bigdotverticesy] at (18,5){};
\node[dotvertices, label=right:{\footnotesize Elements of $M_x\cap M_y$}] at (18,5){};
\node[vertices, label=right:{\footnotesize Generators of $M$}] at (18,9.5){};

\end{tikzpicture}
\end{center}
\caption{Diagram of $M_x, M_y$ for $M=(xy^{12},x^2y^{11},x^5y^{10},x^6y^8,x^{10}y^6,x^{13}y)$}\label{fig:pictureMx}
\end{figure}

Assuming $g_i\neq 0$ and $g_{r+1}\in M_x$ for some syzygy, we may without loss of generality consider only syzygies where $x\nmid g_i$ (as syzygies where $x|g_i$ are in the $S$-span of syzyzgies with $x\cdot e_{f_i}$ and $x^{a_i-1}y^{b_{i+1}-1}\cdot e_{f_{r+1}}$.)
Considering first syzygies with $g_i=y$, we have
$$y\cdot x^{a_i-1}y^b_i-g_{r+1}\cdot y=x^{a_i-1}y^{b_i+1}-g_{r+1}\cdot y=0,$$
giving us that $g_{r+1}=x^{a_i-1}y^{b_i-1}$.

Our minimal syzygies must be of the form
$$\partial_2(g_i\cdot e_{f_i}+x^{a_i-1}y^{b_{i+1}-1}\cdot e_{f_{r+1}})=g_i\cdot x^{a_i-1}y^{b_i}-x^{a_i-1}y^{b_{i+1}}=0.$$
So $g_i=y$ and $g_{r+1}=x^{a_i-1}y^{b_i}$.

Finally, we remove all elements $y^{b_{i+1}-b_i}\cdot e_{f_i}$ from our generating set for syzygies $g_1e_{f_1}+\cdots+g_re_{f_r}+g_{r+1}e_{f_{r+1}}$, by noting that each of these are in the $S$-span of elements $y\cdot e_{f_i}+x^{a_i-1}y^{b_i}\cdot e_{f_{r+1}}$ and $x^{a_i-1}y^{b_{i+1}-1}\cdot e_{f_{r+1}}$:
\begin{align*}
y^{b_{i+1}-b_i}\cdot e_{f_i}&=y^{b_{i+1}-b_i-1}\left(y\cdot e_{f_i}+x^{a_i-1}y^{b_i}\cdot e_{f_{r+1}}-x^{a_i-1}y^{b_i}\cdot e_{f_{r+1}}\right)\\
&=y^{b_{i+1}-b_i-1}\left(y\cdot e_{f_i}+x^{a_i-1}y^{b_i}\cdot e_{f_{r+1}}\right)-y^{b_{i+1}-b_i-1}\left(x^{a_i-1}y^{b_i}\cdot e_{f_{r+1}}\right)\\
&=y^{b_{i+1}-b_i-1}\left(y\cdot e_{f_i}+x^{a_i-1}y^{b_i}\cdot e_{f_{r+1}}\right)-x^{a_i-1}y^{b_{i+1}-1}\cdot e_{f_{r+1}}
\end{align*}
Our third syzygy module is then generated by $e_{c_i^x}$, $e_{c_i^y}$, and $e_{d_i}$, with the described $\partial_3$.
\end{proof}
\begin{proof}[Proof of Case 2]
In Case 2, consider separately the cases where $g_r=0$ and $g_r\neq 0$. If $g_r=0$, then we have the equations
\begin{align}
\left({\displaystyle \sum_{i=1}^{r-1}g_ix^{a_{i}-1}y^{b_i}}\right)-g_{r+1}y&=0\\
g_{r+1}x&=0.
\end{align}
Using Lemma~\ref{lem:syzMx} and recreating the proof on Case 1,  we have that a generating set for all syzygies mapping to $S^{r+1}$ under $\partial_3$ with $g_r=0$ are of the forms:
\begin{align*}
\partial_3(e_{c_i^x})&=x\cdot e_{f_i} \hspace{95pt} \text{ for \(1\leq i\leq r-1\)}\\
\partial_3(e_{c_i^y})&=y\cdot e_{f_i}+x^{a_i-1}y^{b_i}\cdot e_{f_{r+1}} \hspace{10pt}\text{ for \(1\leq i\leq r-1\)}\\
\partial_3(e_{d_i})&=x^{a_i-1}y^{b_{i+1}-1}\cdot e_{f_{r+1}}\hspace{32pt}\text{ for \(1\leq i\leq r-1\)}.
\end{align*}
If $g_r\neq 0$, then as we are resolving a monomial ring over a monomial quotient our syzygies must satisfy either
\begin{enumerate}[(i)]
\item $g_{r+1}=0$, $g_ix^{a_i-1}y^{b_i}=0$, and $g_ry^{b_r-1}=0$, or
\item $g_{r+1}\neq 0$, $g_{r+1}y=0$, and $g_ry^{b_r-1}-g_{r+1}x=0$.
\end{enumerate}
If condition (i) holds, we have the syzygy $g_i=0$ and $g_r=y$ generates all such new syzygies (all syzygies with $g_i\neq 0$ may be rewritten as sums of $y\cdot e_{r}$ and $x\cdot e_{f_i}$ for $1\leq i\leq r-1$.) In the second case we have $g_{r+1}\in M_y$, and for $g_r\neq 0$, we have new minimal syzygy $x\cdot e_{f_r}-y^{b_r-1}\cdot e_{f_{r+1}}.$ The proof that adding these two syzygies to our generating set gives us a complete $S$-spanning set for syzygies is similar to that given in Case 1 and omitted here.

As in Case 1, our third syzygy module has been shown to be generated by $e_{c_i^x}$, $e_{c_i^y}$, and $e_{d_i}$, with the described $\partial_3$.
\end{proof}

A careful examination of the degrees of the map $\partial_4$ produces the appropriate twists for $F_3$ in a graded resolution:

\[F_3=\left(\bigoplus_{1\leq i\leq r} S(-a_i-b_i-1)^2\right)\oplus\left(\bigoplus_{1\leq i\leq r-1} S(-a_i-b_{i+1})\right)\cong S^{3r-1}\] 

\section{Fourth Syzygy Modules in Main Case}

Finally, we construct our fourth syzygy module $F_4$ and fourth syzygy map \(\partial_4\) in a minimal free resolution of \(\kk\) over \(S\).

\begin{proposition}[Fourth Syzygy Module of \(\kk\) over \(S\)]\label{prop:syz4} Let \(M=(x^{a_1}y^{b_1},...,x^{a_r}y^{b_r})\) be a monomial ideal given by its minimal set of generators, ordered so that \(a_1> a_2>\cdots>a_r\geq 0\) and \(0\leq b_1 <b_2<\cdots < b_r\).  Let either \(r>2\) or \(r=2\) and \(M\neq(x^n,y^m)\). The fourth syzygy module \(F_4\cong F_2^r\oplus F_1^{r-1}\) and \(\partial_4=\partial_2^r\oplus\partial_1^{r-1}\). We may choose as a generating set
\[\bigl\{e_{h_j^x},e_{h_j^y}\bigr\}_{j=1}^{r-1}\cup \biggl\{\bigl\{e_{k_{ij}}\bigr\}_{i=1}^{r+1}\biggr\}_{j=1}^r.\]
\begin{enumerate}[{\bf Case }1:]
\item If all generators of \(M\) are divisible by \(x\) (i.e. \(a_r\geq 1\)), then the fourth syzygy map \(\partial_4\) is given by
\begin{align*}
\partial_4(e_{h_j^x})&=x\cdot e_{d_j} \hspace{95pt} \text{ for \(1\leq j\leq r-1\)}\\
\partial_4(e_{h_j^y})&=y\cdot e_{d_j} \hspace{95pt} \text{ for \(1\leq j\leq r-1\)}\\
\partial_4(e_{k_{ij}})&=\begin{cases}
x^{a_i-1}y^{b_i}\cdot e_{c_{j}^x} &\hspace{20pt}\text{ for \(1\leq i,j\leq r\)}\\
-y\cdot e_{c_j^x}+x\cdot e_{c_j^y} &\hspace{20pt}\text{ if $i=r+1$ and \(1\leq j\leq r\)}.
\end{cases}
\end{align*}
\item If the final generator of \(M\) is \(m_r=y^{b_r}\), then the fourth syzygy map \(\partial_4\) is given by
\begin{align*}
\partial_4(e_{h_j^x})&=x\cdot e_{d_j} \hspace{95pt} \text{ for \(1\leq j\leq r-1\)}\\
\partial_4(e_{h_j^y})&=y\cdot e_{d_j} \hspace{95pt} \text{ for \(1\leq j\leq r-1\)}\\
\partial_4(e_{k_{ij}})&=\begin{cases}
x^{a_i-1}y^{b_i}\cdot e_{c_{j}^x} &\hspace{20pt}\text{ for \(1\leq i\leq r-1\) and \(1\leq j\leq r\)}\\
y^{b_r-1}\cdot e_{c_{j}^y} &\hspace{20pt}\text{ for \(i=r\) and \(1\leq j\leq r\)}\\
-y\cdot e_{c_j^x}+x\cdot e_{c_j^y} &\hspace{20pt}\text{ if $i=r+1$ and \(1\leq j\leq r\)}.
\end{cases}
\end{align*}
\end{enumerate}
\end{proposition}

\begin{proof}[Proof of Case 1]
From Proposition~\ref{prop:syz3}, our third syzygy module $F_3$ had generators
\[\bigl\{e_{c_i^x}\bigr\}_{i=1}^r\cup\bigl\{e_{c_i^y}\bigr\}_{i=1}^r\cup\bigl\{e_{d_i}\bigr\}_{i=1}^{r-1},\]
where the $\partial_3$ was given by
\begin{align*}
\partial_3(e_{c_i^x})&=x\cdot e_{f_i} \hspace{95pt} \text{ for \(1\leq i\leq r\)}\\
\partial_3(e_{c_i^y})&=y\cdot e_{f_i}+x^{a_i-1}y^{b_i}\cdot e_{f_{r+1}} \hspace{10pt}\text{ for \(1\leq i\leq r\)}\\
\partial_3(e_{d_i})&=x^{a_i-1}y^{b_{i+1}-1}\cdot e_{f_{r+1}}\hspace{32pt}\text{ for \(1\leq i\leq r-1\)}.
\end{align*}
Beginning with syzygies on the $e_{d_j}$ terms, we note that $x\cdot x^{a_j-1}y^{b_{j+1}-1}=y\cdot x^{a_j-1}y^{b_{j+1}-1}=0\in S$, so we have minimal generators $\{x\cdot e_{d_j},y\cdot e_{d_j}\}_{j=1}^{r-1}$. We now find a generating set for remaining syzygies \emph{only} involving the $e_{c_j^x},e_{c_j^y}$.
Given a syzygy $g_1^x\cdot e_{c_1^x}+g_1^y\cdot e_{c_1^y}+\cdots+g_r^y\cdot e_{c_r^x}+g_r^y\cdot e_{c_r^y}$ for $g_i^x,g_i^y\in S$, we must have
\begin{align*}
g_1^x(x\cdot e_{f_1})&+g_1^y(y\cdot e_{f_1}+x^{a_1-1}y^{b_1}\cdot e_{f_{r+1}})+\cdots\\
&+g_r^x(x\cdot e_{f_1})+g_r^y(y\cdot e_{f_r}+x^{a_r-1}y^{b_r}\cdot e_{f_{r+1}})=0,
\end{align*}
So we have the relations
\begin{align*}
g_1^xx&+g_1^yy=0,\\
g_2^xx&+g_2^yy=0,\\
\vdots \\
g_r^xx&+g_r^yy=0,\text{ and}\\
\sum_{i=1}^rg_i^y &x^{a_i-1}y^{b_i}=0.
\end{align*}
Note that all syzygies of the field over a monomial quotient are generated by either $m\cdot e_{\ast}$ for some generator $e_{\ast}\in F_3$ or by pairs $m_1\cdot e_{\ast,1}+m_2\cdot e_{\ast,2}$ for $e_{\ast,1},e_{\ast,2}$ generators of $F_3$. As $g_j^xx+g_j^yy=0$ must hold for all $j$, it will suffice to consider syzygies on pairs $e_{c_j^x}$ and $e_{c_j^x}$.

Fix a $j\in\{1,...,r\}$ and consider syzygies on $\{e_{c_j^x},e_{c_j^y}\}$. The equalities $g_j^x x+g_j^yy=0$ and $g_j^y x^{a_j-1}y^{b_j}=0$ hold. From the second equality, we must have $g_j^y=0$, $g_j^y=x$, or $g_j^y=y^{b_{j+1}-b_j}$ for a degree minimal relation. As $g_x^xx+y^{b_{j+1}-b_j+1}\neq 0$ for any $g_j^x$ (by $y^d\not\in M$ for any $d\in\NN$ in Case 1,) we must have either $g_j^y=0$ or $g_j^y=x$.

If $g_j^y=0$, then $g_j^x=0\in S$, or $g_j^x\in M_x$. This gives syzygies of the form $g_j^x=x^{a_i-1}y^{b_i}$ and $g_j^x=0$ for all $1\leq i\leq r$. If $g_j^y=x$, then $g_j^xx+xy=0$. The syzygy with $g_j^x=-y$ and $g_j^y=x$ is a minimal generator for all such syzygies.

So $\partial_4$ is of the form given in the statement of the theorem. Note that for each $1\leq j\leq r-1$, we have
\begin{align*}
F_4|_{\{e_{k_j^x},e_{h_j^y}\}}&\cong F_1\\
\partial_4|_{\{e_{h_j^x},e_{h_j^y}\}}&\cong \partial_1,
\end{align*}
and for each $1\leq j\leq r$ we have
\begin{align*}
F_4|_{\{e_{k_{ij}}\}_{i=1}^{r+1}}&\cong F_2\\
\partial_4|_{\{e_{k_{ij}}\}_{i=1}^{r+1}}&\cong \partial_2.
\end{align*}
So we have $F_4=F_1^{r-1}\oplus F_2^r$ and $\partial_4=\partial_1^{r-1}\oplus\partial_2^r$, completing our proof.
\end{proof}

\begin{proof}[Proof of Case 2]
In Case 2, the map $\partial_3$ is given by
\begin{align*}
\partial_3(e_{c_i^x})&=\begin{cases} x\cdot e_{f_i} &\hspace{15pt}\text{for \(1\leq i\leq r-1\)}\\
x\cdot e_{f_r}-y^{b_r-1}\cdot e_{f_{r+1}} &\hspace{15pt}\text{for \(i=r\),}
\end{cases}\\
\partial_3(e_{c_i^y})&=\begin{cases}y\cdot e_{f_i}+x^{a_i-1}y^{b_i}\cdot e_{f_{r+1}} &\hspace{3pt} \text{for \(1\leq i\leq r-1\)}\\
y\cdot e_{f_{r}} &\hspace{3pt} \text{for \(i=r\),}
\end{cases}\\
\partial_3(e_{d_i})&=x^{a_i-1}y^{b_{i+1}-1}\cdot e_{f_{r+1}} \hspace{43pt}\text{ for \(1\leq i\leq r-1\)}.
\end{align*}
The proof in Case 2 is identical to that in Case 1 for $\partial_4(e_{h_j^x})$ and $\partial_4(e_{h_j^y})$. When constructing $\partial_4(e_{k_{ij}})$ in Case 2, for any syzygies we have the equalities
\begin{align*}
g_1^xx+g_1^yy&=0,\\
g_2^xx+g_2^yy&=0,\\
\vdots \\
g_{r}^xx+g_{r}^yy&=0,\text{ and}\\
\left(\sum_{i=1}^{r-1}g_i^y x^{a_i-1}y^{b_i}\right)+g_r^xy^{b_r-1}&=0.
\end{align*}
After fixing a $j\in\{1,...,r\}$ and considering syzygies on pairs $\{e_{c_j^x},e_{c_j^y}\}$, we have minimal syzygies of the forms
\[\left\{x^{a_i-1}y^{b_i}\cdot e_{c_j^x}\right\}_{i=1}^{r-1}\bigcup\left\{y^{b_r-1}\cdot e_{c_j^y}\right\}\bigcup\left\{-y\cdot e_{c_j^x}+x\cdot e_{c_j^y}\right\},\]
via arguments similar to those in Case 1. As in Case 1, we now have that $\partial_4$ is of the form given in the statement of the theorem, with $F_4\cong F_1^{r-1}\oplus F_2^r$ and $\partial_4=\partial_1^{r-1}\oplus\partial_2^r$.
\end{proof}

\section{Proof of Main Theorem and Corollary~\ref{cor:main}}

We now return to the main proof of our theorem.

\mainthm

\begin{proof}
From Propositions~\ref{prop:syz1and2},~\ref{prop:syz3}, and~\ref{prop:syz4}, we  have that the first four stages of the resolution of $\kk$ over $S$ may be written as
\begin{equation*}
{\mathcal F}:\;\;\;\kk\longleftarrow S\longleftarrow F_1\longleftarrow F_2\longleftarrow F_3\longleftarrow\mathop{\oplus}^{\displaystyle F_2^r}_{\displaystyle F_1^{r-1}}.
\end{equation*}
We prove inductively that we may decompose the $F_{i}$ syzygy module in our resolution to a direct sum of the form $F_{1}^{u_i}\oplus F_{2}^{v_i}\oplus F_3^{w_i}$.

For $F_4$, we have by Proposition~\ref{prop:syz4} that $F_4=F_1^{r-1}\oplus F_2^{r}\oplus F_3^0$. Let $F_i=F_1^{u_i}\oplus F_2^{v_i}\oplus F_3^{w_i}$ be the $i^{\text{th}}$-syzygy module of $\kk$ over $S$ for $i\geq 4$. Then a resolution of each component of the direct sum is given by the modules and maps
\begin{align*}
F_1^{u_i}&\xleftarrow{\;\;(\partial_2)^{u_i}\;\;} F_2^{u_i}\\
F_2^{v_i}&\xleftarrow{\;\;(\partial_3)^{v_i}\;\;} F_3^{v_i}\\
F_3^{w_i}&\xleftarrow{\;\;(\partial_1)^{(r-1)w_i}\oplus(\partial_2)^{rw_i}\;\;} F_1^{(r-1)w_i}\oplus F_2^{rw_i},
\end{align*}
so the $(i+1)^{\text{st}}$ syzygy module is $F_{i+1}=F_1^{(r-1)w_i}\oplus F_2^{u_i+rw_i}\oplus F_3^{v_i}$.
\end{proof}

\maincor

\begin{proof}
From Propositions~\ref{prop:syz1and2} and \ref{prop:syz3}, we have that 
\begin{align*}
\beta_2^S(\kk)&=r+1=2+(r-1)\text{ and}\\
\beta_3^S(\kk)&=3r-1= (r+1)+2\cdot(r-1).
\end{align*}
To calculate $\beta_i^S(\kk)$ for $i\geq 4$, we note that $F_4=F_2^r\oplus F_1^{r-1}$. So the $i^{\text{th}}$ stage of the resolution for $i\geq 4$ may be rewritten as $F_i=F_{i-2}^{r}\oplus F_{i-3}^{r-1}$, giving us a third-order linear recursion formula,
\[\beta_i^S(\kk)=r\beta_{i-2}^S(\kk)+(r-1)\beta_{i-3}^S(\kk),\]
or equivalently, the second-order linear recursion,
\[\beta_i^S(\kk)=\beta_{i-1}^S(\kk)+(r-1)\beta_{i-2}^S(\kk).\]

Our Poincar\'{e}-Betti series then is
\[P_S(z)=\frac{(1+z)^2}{1-rz^2+(1-r)z^3}=\frac{1+z}{1-z+(1-r)z^2}.\]
\end{proof}

\section{Conclusions and Future Work}
Between Theorem~\ref{thm:main} and Appendix~\ref{sec:degenerate}, we have a complete classification of all resolutions of $\kk$ over monomial quotient rings $S=\kk[x,y]/M$. Note that interestingly, the recursion formula appearing in the actual construction of syzygy modules \(F_i\) is
\begin{equation*}
F_i=(F_{i-2})^r\oplus(F_{i-3})^{r-1}\text{ if \(i\geq 4\).}
\end{equation*}
This third-order linear recursion formula matches the Poincar\'{e}-Betti series denominator calculated in~\cite{MR2188858},
\[P_S(z)=\frac{(1+z)^2}{1-rz^2+(1-r)z^3}=\frac{(1+z)^2}{b_S(z)},\]
where $b_S(z)$ has been interpreted combinatorially in terms of dimensions of the homologies of intervals in the lcm-lattice of $M$.

Whether all resolutions of $\kk$ over $S=\kk[x_1,...,x_n]/M$ for monomial ideals $M$ generated in degree 2 or higher, not all pure powers, can be decomposed via a similar recursive formula for $F_i$ is unknown. All monomial ideals $M$ in $\kk[x,y,z]$ and $\kk[x,y,z,w]$ that the author has examined so far, however, have had decompositions of their syzygy modules $F_i$ matching the formula $b_S(z)$ for the denominator of the Poincar\'{e}-Betti series produced in~\cite{MR2188858} and~\cite{MR2290910}.

\begin{openquestion}[Resolutions of $\kk$ over Monomial Quotient Rings] Let $M=(m_1,...,m_n)\subset R=\kk[x_1,...,x_k]$ with $\deg(m_i)\geq 2$ for all $i$ and \emph{not} all $m_i$ pure powers of the variables. Let $b_S(z)$ be the Poincar\'{e}-Betti denominator of $\kk$ over $S=$ produced in~\cite{MR2188858} and~\cite{MR2290910}, setting $b_S(z)=1-c_2z^2-\cdots-c_dz^d$. Does there exist an $n\geq d+1$ and a decomposition of $F_i$ such that $F_i=\left(F_{i-2}\right)^{c_2}\oplus\cdots\oplus \left(F_{i-d}\right)^{c_d}$ for all $i\geq n$?
\end{openquestion}

Thanks to Brandon Stone and Courtney Gibbons for noting that the examples $M=(x^{\alpha},y^{\beta},z^{\gamma})$ for $\alpha,\beta,\gamma\geq 1$ are degenerate cases similar to those in Appendix~\ref{sec:degenerate}. In particular, this conjecture excludes complete intersections $M=(x_{i_1}^{\alpha_1},x_{i_2}^{\alpha_2},...,x_{i_r}^{\alpha_r})\subseteq \kk[x_1,...,x_k]=R$ and $S=R/M$, as the Poincar\'{e}-Betti series of $\kk$ over $S$ in these cases are \emph{not} of the form $b_S(z)=1-c_2z^2-\cdots-c_dz^d.$

\section*{Acknowledgements} 
The author would like to thank James Parson for many interesting conversations on infinite free resolutions and Poincar\'{e} series, and to thank Neil Epstein for initially suggesting this problem. The author was partially supported by a Hood College Board of Associates (BOA) grant for Summer 2013.

\appendix

\section{Appendix: Resolutions in Degenerate Cases}\label{sec:degenerate}
The five degenerate cases in resolutions of \(k\) over \(S=\kk[x,y]/M\) are described here, along with the the syzygies modules and the attaching maps of resolutions of \(\kk\) over \(S=\kk[x,y]/M\) in each case. Proofs are omitted.

\begin{enumerate}[{Type }I:]
\item \(M=(x)\) (or \(M=(y)\)),
\[{\mathcal F}:\;\;\;\kk\longleftarrow S\xleftarrow{\;\;\;\;(x)\;\;\;\;}S(-1)\longleftarrow 0.\]
\item \(M=(x^ay^b)\) where \(a+b\geq 2\),
\begin{align*}{\mathcal F}:\;\;\;\kk\longleftarrow S&\xleftarrow{\;\;(x\;y)\;\;} S^2\xleftarrow{\begin{pmatrix}
-y & x^{a-1}y^b\\
x &0
\end{pmatrix}}S^2\xleftarrow{\begin{pmatrix}
x^{a-1}y^b&0\\
y &x
\end{pmatrix}}S^2\\
&\xleftarrow{\begin{pmatrix}
-x&0\\
y &x^{a-1}y^{b}
\end{pmatrix}}S^2\xleftarrow{\begin{pmatrix}
x^{a-1}y^b&0\\
y&x
\end{pmatrix}}S^2\longleftarrow\cdots,
\end{align*}
with $\partial_i=\partial_{i-2}$ for all $i\geq 5$. (If $a=0$, interchange $x$ and $y$.)
\item \(M=(x,y)\),
\[{\mathcal F}:\;\;\;\kk\longleftarrow S\longleftarrow 0.\]
\item \(M=(x^a,y)\) (or \(M=(x,y^b)\)) where \(a,b\geq 2\),
\[{\mathcal F}:\;\;\;\kk\longleftarrow S\xleftarrow{\;\;\;\;(x)\;\;\;\;}S\xleftarrow{\;\;\;\;(x^{a-1})\;\;\;\;}S\xleftarrow{\;\;\;\;(x)\;\;\;\;}S\longleftarrow \cdots,\]
with $\partial_i=\partial_{i-2}$ for all $i\geq 3$
\item \(M=(x^a,y^b)\) where \(a,b\geq 2\).
\end{enumerate}
\[{\mathcal F}:\;\;\;\kk\longleftarrow S\xleftarrow{\;\;\;\;(x\;y)\;\;\;\;}S^2\xleftarrow{\;\;\;\;\partial_2\;\;\;\;}S^3\xleftarrow{\;\;\;\;\partial_3\;\;\;\;}S^4\xleftarrow{\;\;\;\;\partial_4\;\;\;\;}S^5\longleftarrow\cdots,\]
where we give a formula for $\partial_i$ for $i\geq 2$ via polynomials $f_j^{(i)},g_j^{(i)}$ defined inductively:

\noindent {\bf Case ($i=2$):} Let $\left\{e_1^{(1)},e_2^{(1)}\right\}$ be a generating set for $F_1\cong S^2$ and $\left\{e_1^{(2)},e_2^{(2)},e_3^{(2)}\right\}$ be a generating set for $F_2\cong S^3$. Set syzygy map $\partial_2$ to be:
\begin{align*}
\partial_2(e_1^{(2)})&=x^{a-1}\cdot e_1^{(1)}\\
\partial_2(e_2^{(2)})&=y^{b-1}\cdot e_2^{(1)}\\
\partial_2(e_3^{(2)})&=-y\cdot e_1^{(1)}+x\cdot e_2^{(1)}.
\end{align*}
Set $f_1^{(2)}=x^{a-1}$, $f_2^{(2)}=y^{b-1}$, $g_3^{(2)}=-y$, and $f_3^{(2)}=x$.

\vspace{10pt}

\noindent {\bf Case ($i>2$):} Let $\{e_j^{(k)}\}_{j=1}^{k+1}$ be a generating set for the $k^{\text{th}}$ syzygy module $F_k\cong S^{k+1}$ for $k<i$, where we have chosen polynomials $\left\{f_j^{(k)}\right\}_{j=1}^{k+1}$ and $\left\{g_j^{(k)}\right\}_{j=3}^{k+1}$ so that our syzygy maps $\partial_k$ are:
\begin{align*}
\partial_k(e_j^{(k)})=\begin{cases}
f_j^{(k)}\cdot e_j^{(k-1)}&\text{ if $j=1,2$}\\
g_j^{(k)}\cdot e_{j-2}^{(k-1)}+f_j^{(k)}\cdot e_{j}^{(k-1)}&\text{ if $3\leq j\leq k$}\\
g_{k+1}^{(k)}\cdot e_{k-1}^{(k-1)}+f_{k+1}^{(k)}\cdot e_{k}^{(k-1)}&\text{ if $j=k+1$.}
\end{cases}
\end{align*}
Then the syzygy module $F_i\cong S^{i+1}$ has basis $\left\{e_j^{(i)}\right\}_{j=1}^{i+1},$ with syzygy map $\partial_i$ given by:
\begin{align*}
\partial_i(e_j^{(i)})=\begin{cases}
f_j^{(i)}\cdot e_j^{(i-1)}&\text{ if $j=1,2$}\\
g_j^{(i)}\cdot e_{j-2}^{(i-1)}+f_j^{(i)}\cdot e_{j}^{(i-1)}&\text{ if $3\leq j\leq i$}\\
g_{i+1}^{(i)}\cdot e_{i-1}^{(i-1)}+f_{i+1}^{(i)}\cdot e_{i}^{(i-1)}&\text{ if $j=i+1$,}
\end{cases}
\end{align*}
where the $f_j^{(i)}$, $g_j^{(i)}$ are defined inductively by:
\begin{align*}
f_1^{(i)}&=\frac{x^a}{f_1^{(i-1)}},\\
f_2^{(i)}&=\frac{y^b}{f_2^{(i-1)}},\\
g_j^{(i)}&=g_{j}^{(i-1)},\;\;\;\;f_j^{(i)}=-f_{j-2}^{(i-1)}\text{ for $3\leq j\leq i$}\\
g_{i+1}^{(i)}&=f_{i}^{(i-1)},\;\;\;\;f_{i+1}^{(i)}=-f_{i-1}^{(i-1)}.
\end{align*}

\begin{example} Let $M=(x^3,y^7)$ be an ideal in $R=\kk[x,y]$. Then the first two stages of the resolution of $\kk$ over $S=R/M$ are given by:
\[{\mathcal F}:\;\;\;\kk\longleftarrow S\xleftarrow{\;\;\;\;(x\;y)\;\;\;\;}S^2\xleftarrow{\;\;\;\;\partial_2\;\;\;\;}S^3\xleftarrow{\;\;\;\;\partial_3\;\;\;\;}\cdots,\]
with $\partial_2=\begin{pmatrix}
x^{2}&0 & -y\\
0 & y^{7}&x
\end{pmatrix},$ giving us $f_1^{(2)}=x^2$, $f_2^{(2)}=y^6$, $g_3^{(2)}=-y$ and $f_3^{(2)}=x$. From our inductive formula above, we calculate $f_j^{(3)}$ and $g_j^{(3)}$:
\begin{align*}
f_1^{(3)}&=\frac{x^3}{f_1^{(2)}}=\frac{x^3}{x^2}=x,\\
f_2^{(3)}&=\frac{y^6}{f_2^{(2)}}=\frac{y^6}{x^5}=y,\\
g_3^{(3)}&=g_{3}^{(2)}=-y,\;\;\;\;f_3^{(3)}=-f_{1}^{(2)}=-x^2\\
g_4^{(3)}&=f_3^{(2)}=x,\;\;\;\;f_4^{(3)}=-f_2^{(2)}=-y^6.
\end{align*}
Putting this into our formula for $\partial_3$, we have that the third stage of the resolution should be spanned by generators $\left\{e_1^{(3)},e_2^{(3)},e_3^{(3)},e_4^{(3)}\right\}$, where
\begin{align*}
\partial_3(e_1^{(3)})&=x\cdot e_1^{(2)}\\
\partial_3(e_2^{(3)})&=y\cdot e_2^{(2)}\\
\partial_3(e_3^{(3)})&=-y\cdot e_1^{(2)}-x^2\cdot e_3^{(3)}\\
\partial_3(e_4^{(3)})&=x\cdot e_2^{(2)}-y^6\cdot e_3^{(3)},
\end{align*}
or
\[\partial_3=\begin{pmatrix}
x & 0 & -y & 0\\
0 & y & 0 & x\\
0 & 0 & -x^2 & -y^6\\
\end{pmatrix}.\]

\end{example}

\providecommand{\bysame}{\leavevmode\hbox to3em{\hrulefill}\thinspace}
\providecommand{\MR}{\relax\ifhmode\unskip\space\fi MR }
\providecommand{\MRhref}[2]{%
  \href{http://www.ams.org/mathscinet-getitem?mr=#1}{#2}
}
\providecommand{\href}[2]{#2}


\end{document}